\documentclass[a4paper,12pt,reqno]{amsart} % reqno = right equation numbers

\makeatletter
\tagsleft@false % make sure tags (equation numbers) are on the right
\makeatother

\usepackage{comment}

\usepackage{tikz}
\usepackage{pgfplots}
\pgfplotsset{compat=1.18}
\usepackage{longtable}
\usepackage{amsmath}
\usepackage{amssymb}
\usepackage{amsthm}
\usepackage{mathrsfs}
\usepackage{ifthen}
\usepackage{graphicx}
\usepackage[T1]{fontenc}

\numberwithin{equation}{section}

\newtheorem*{theorem*}{Theorem}

\newtheorem{thm}{Theorem}[section]
\newtheorem{lem}{Lemma}[section]

\newtheorem*{cor*}{Corollary}

\theoremstyle{definition}

\newtheorem{prob}{Problem}[section]

\newtheorem{cor}{Corollary}[section]

\newcounter{minutes}
\divide\time by 60
\newcounter{hours}
\multiply\time by 60
\addtocounter{minutes}{-\time}

\newcounter{own}
\def\theown{\thesection.\arabic{own}}

\newenvironment{pf}[1][]{%
	\vskip 3mm
	\noindent
	\ifthenelse{\equal{#1}{}}%
	{{\slshape Proof. }}%
	{{\slshape #1.} }%
}%
{\qed\bigskip}

\newcounter{alphabet}

\def\be{\begin{equation}}
	\def\ee{\end{equation}}

\newcommand{\bee}{\begin{enumerate}}
	\newcommand{\eee}{\end{enumerate}}

\newcommand{\blem}{\begin{lem}}
	\newcommand{\elem}{\end{lem}}
\newcommand{\bthm}{\begin{thm}}
	\newcommand{\ethm}{\end{thm}}
\newcommand{\bcor}{\begin{cor}}
	\newcommand{\ecor}{\end{cor}}
\newcommand{\eeg}{\end{examp}}

\newcommand{\eegs}{\end{examples}}
\newcommand{\edefe}{\end{defin}}
\newcommand{\bprob}{\begin{prob}}
\newcommand{\eprob}{\end{prob}}
\newcommand{\bei}{\begin{itemize}}
\newcommand{\eei}{\end{itemize}}

\usepackage{xcolor}

\begin{document}

\title{{The second and third Hankel determinants for certain classes of functions}}

\author{Vasudevarao Allu}
\address{Vasudevarao Allu,
Department of Mathematics,
School of Basic Sciences,
Indian Institute of Technology Bhubaneswar,
Bhubaneswar-752050, Odisha, India.}
\email{avrao@iitbbs.ac.in}

\author{Shobhit Kumar}
\address{Shobhit Kumar,
Department of Mathematics,
School of Basic Sciences,
Indian Institute of Technology Bhubaneswar,
Bhubaneswar-752050, Odisha, India.}
\email{a21ma09007@iitbbs.ac.in}



\def\thefootnote{}
\footnotetext{
{ }
}
\makeatletter\def\thefootnote{\@arabic\c@footnote}\makeatother
\begin{abstract}
Let $\mathcal{A}$ denote the class of analytic functions such that $f(0)=0$ and $f'(0)=1$ in the unit disk $\mathbb{D}:=\{z \in \mathbb{C}: |z|<1\}.$
In the present paper, we consider  $\mathcal{C}(\varphi) := \left\{ f \in \mathcal{A} : 1+zf''(z)/f'(z) \prec \varphi(z):=(1+z/2)^2 \right\}$, as subclass of convex functions and  compute the sharp second and third Hankel determinants for functions in $\mathcal{C}(\varphi)$. 
\end{abstract}
\maketitle
\markboth{V. Allu and S. Kumar}{Convex Ma--Minda subclass of $\varphi(z)$}
\setcounter{page}{1}
\section{Introduction}
Let $\mathcal{A}$ denote the class of analytic functions $f$ in the unit disk
$
\mathbb{D}:=\{z\in\mathbb{C}:|z|<1\}
$
normalized by
$
f(0)=0$ and  $f'(0)=1$.
Thus each function $f\in\mathcal{A}$ has a power series representation such that
\[
f(z)=z+\sum_{n=2}^{\infty}a_n z^n,
\qquad z\in\mathbb{D}.
\]
Let $\mathcal{S}$ denote the subclass of $\mathcal{A}$ consisting of univalent functions.
A function $f\in \mathcal{A}$ is said to be starlike if $f(\mathbb{D})$ is a starlike domain. We denote by $\mathcal{S}^*$ the class of starlike functions $f$ in $\mathcal{S}$. It is well known that a function $f\in \mathcal{A}$ belongs to $\mathcal{S}^*$ if, and only if,
\[
\operatorname{Re}\!\left(\frac{z f'(z)}{f(z)}\right)>0,
\qquad z\in\mathbb{D}.
\]
Further, a domain $\Omega\subseteq \mathbb{C}$ is called convex if the line segment joining any two points of $\Omega$ lies entirely in $\Omega$. A function $f\in \mathcal{A}$ is called convex if $f(\mathbb{D})$ is a convex domain. We denote by $\mathcal{C}$ the class of convex functions in $\mathcal{S}$. A function $f\in \mathcal{A}$ belongs to $\mathcal{C}$ if, and only if,
\[
\operatorname{Re}\!\left(1+\frac{z f''(z)}{f'(z)}\right)>0,
\qquad z\in\mathbb{D}.
\]
These classical families and their geometric consequences form a central part of geometric function theory; see, for instance, Goodman \cite{Goodman1983}.\\[2mm]
In order to unify and generalize many well-known subclasses of starlike and convex functions, Ma and Minda \cite{MaMinda1992} introduced a systematic approach based on subordination. Let $\varphi$ be analytic in $\mathbb{D}$ with $\operatorname{Re}\,\varphi(z)>0$, normalized by $\varphi(0)=1$ and $\varphi'(0)>0$, and assume that $\varphi(\mathbb{D})$ is symmetric with respect to the real axis and starlike with respect to the point $1$. Then a function $f\in\mathcal{A}$ is said to belong to the Ma--Minda starlike class $\mathcal{S}^*(\varphi)$  if, and only if,
\[
\frac{z f'(z)}{f(z)} \prec \varphi(z),
\qquad z\in\mathbb{D},
\]
and to the Ma--Minda convex class $\mathcal{C}(\varphi)$ if, and only if,
\[
1+\frac{z f''(z)}{f'(z)} \prec \varphi(z),
\qquad z\in\mathbb{D}.
\]
The Ma--Minda framework contains many important subclasses as special cases and provides a flexible setting for studying coefficient problems, radius problems, and extremal questions, (see \cite{Janowski1971, MaMinda1992}).\\[2mm]
A substantial amount of work has been devoted to Ma--Minda type classes associated with specific target regions such as conic domains, vertical strip domains, cardioid-type regions, and domains generated by special functions. For example, conic regions and $k$-uniform convexity were investigated by Kanas and Wisniowska in \cite{KanasWisniowska1999}, cardioid-type starlikeness was studied by Kumar and Kamaljeet in \cite{KumarKamaljeet2021} and by Sharma, Jain and Ravichandran in \cite{SharmaJainRavichandran2016}, and several subclasses connected with vertical strip domains were examined by Sun, Wang, Rasila and Soko{\l} in \cite{SunWangRasilaSokol2019}, by Kwon and Sim in \cite{KwonSim2019}, and by Mahzoon and Soko{\l} in \cite{MahzoonSokol2023}. Further examples of Ma--Minda type starlike families associated with special mappings include several authors (see  \cite{GoelKumar2020}, \cite{KurokiOwa2012}, \cite{MendirattaNagpalRavichandran2015}, \cite{UralegaddiGanigiSarangi1994}).\\[2mm]
Along with geometric properties, coefficient functionals play an important role in the study of these classes. Among them, Hankel determinants provide refined information about the growth and interaction of the Taylor coefficients. For a function
$f\in \mathcal{S},$
the second Hankel determinant is defined by
\[
H_2(2)=a_2 a_4-a_3^2,
\]
and the third Hankel determinant is given by
\[
H_3(1)=
\begin{vmatrix}
	1 & a_2 & a_3\\
	a_2 & a_3 & a_4\\
	a_3 & a_4 & a_5
\end{vmatrix}.
\]
The systematic study of Hankel determinants for univalent functions goes back to Pommerenke \cite{Pommerenke1966}. Sharp bounds for $H_2(2)$ and $H_3(1)$ have subsequently been obtained for several subclasses. For instance, sharp estimates for both the second and third Hankel determinants were derived for the class $\mathcal{S}_L$ by Banga and Kumar \cite{BangaKumar2020}, and radius problems in $\mathcal{S}_L$ were investigated by Soko{\l} \cite{Sokol2009}. In addition, inverse coefficient problems and related coefficient estimates form a classical theme in the theory of convex and $(\alpha,\beta)$-convex functions; see Libera and Z{\l}otkiewicz \cite{LiberaZlotkiewicz1982} and Prokhorov and Szynal \cite{ProkhorovSzynal1981}. Recent contributions to Hankel-type problems for inverse and logarithmic inverse coefficients may be found in Allu and Shaji \cite{VasuAmal2025}. We also mention that sharper coefficient control inside the Carath\'eodory class has been studied, for example, via higher coefficient problems in Kwon, Lecko and Sim \cite{KwonLeckoSim2018}, which is useful in handling extremal questions involving $a_4$ and $a_5$.\\[2mm]
Motivated by these developments,in this paper  we study  the Ma--Minda convex subclass generated by
\[
\varphi(z)=1+z+\frac{z^2}{4}.
\]
For functions
\[
f(z)=z+\sum_{n=2}^{\infty}a_n z^n
\]
belonging to this class, we obtain sharp estimates for the second Hankel determinant $H_2(2)$ and the third Hankel determinant $H_3(1)$.
%%%%%%%%%%%%
%%%%%%%%%%%%%%%%
\section{Certain properties of the class $\mathcal{C}(\varphi)$}
In this section we prove that the function $\varphi(z)$
satisfy all the properties for Ma-Minda function and we also provide the integral representation for the functions in the class $\mathcal{C}(\varphi).$ It is easy to see that $\varphi(0)=1$ and $\varphi'(0)>0$, and that $\varphi$ is symmetric with respect to the real axis. Since the domain under consideration is the unit disk $\mathbb{D}$, for $z\in \overline{\mathbb{D}}$ we have
\[
\left|1+\frac{z}{2}\right|\le 1+\frac{|z|}{2}\le \frac{3}{2}
\]
and
\[
\left|1+\frac{z}{2}\right|\ge 1-\frac{|z|}{2}\ge \frac{1}{2}.
\]
Hence
\begin{align}\label{eq:2.1}
\frac{1}{4} \;\le \left|\left(1+\frac{z}{2}\right)^{2}\right| \;\le \frac{9}{4}
\qquad  \text{for all } z \in \overline{\mathbb{D}}.
\end{align}
To check the univalency of $\varphi(z)$, assume \( \varphi(z_1) = \varphi(z_2) \). Then
\[
\left(1 + \frac{z_1}{2} \right)^2 = \left(1 + \frac{z_2}{2} \right)^2.
\]
Hence
\[
1 + \frac{z_1}{2} = \pm \left(1 + \frac{z_2}{2} \right).
\]
In the first case, we have
\[
1 + \frac{z_1}{2} = 1 + \frac{z_2}{2},
\]
which implies $z_{1}=z_{2}.$
While in the second case, we have
\[
1 + \frac{z_1}{2} = -\left(1 + \frac{z_2}{2} \right),
\]
which implies
\[
1 + \frac{z_1}{2} = -1 - \frac{z_2}{2},
\qquad\text{hence}\qquad
z_1 + z_2 = -4.
\]
However, for $z_1,$ $z_2\in \mathbb{D},$ we have
\[
|z_1 + z_2| \leq |z_1| + |z_2| < 2,
\]
which contradicts \( z_1 + z_2 = -4 \). Therefore \( \varphi(z) \) is univalent in  $\mathbb{D}$.\\[2mm]
To verify the starlikeness of the function, recall that
\begin{align*}
\varphi(z) = \left(1 + \frac{z}{2} \right)^2.
\end{align*}
Then, we have
\begin{align*}
\varphi(0) = 1,
\qquad
\varphi'(z) = 1 + \frac{z}{2}.
\end{align*}
To verify that \( \varphi \) is starlike with respect to \( \varphi(0) = 1 \), we show that
\[
\operatorname{Re}\!\left( \frac{z \varphi'(z)}{\varphi(z) - 1} \right) > 0 
\quad \text{for all } z \in \mathbb{D}.
\]
Since,
\[
\varphi(z)-1 = z + \frac{z^2}{4},
\]
we obtain, 
\[
\frac{z\,\varphi'(z)}{\varphi(z)-1}
=
\frac{z(1+z/2)}{z+z^{2}/4}
=
\frac{1+z/2}{1+z/4}.
\]
Therefore, writing $z=re^{i\theta}$ with $0\le r<1$, we get
\[
\operatorname{Re}\,\left(\frac{1+z/2}{1+z/4}\right)
=
\frac{\operatorname{Re}\!\left((1+z/2)\left(1+\overline{z}/4\right)\right)}{|1+z/4|^{2}}
=
\frac{1+\frac34\operatorname{Re}z+\frac18|z|^{2}}{|1+z/4|^{2}}.
\]
Moreover, since $\operatorname{Re}z\ge -|z|$, we get
\[
1+\frac34\operatorname{Re}z+\frac18|z|^{2}
\ge
1-\frac34|z|+\frac18|z|^{2}
>
1-\frac34+\frac18
=
\frac{3}{8}
>0,
\]
and therefore
\[
\operatorname{Re}\!\left(\frac{z\,\varphi'(z)}{\varphi(z)-1}\right)
=
\operatorname{Re}\!\left(\frac{1+z/2}{1+z/4}\right) > 0
\quad \text{for all } z \in \mathbb{D}.
\]
This shows that the function is starlike with respect to the point $\varphi(0)=1$.\\[2mm]
Now we verify that $\operatorname{Re}\,\varphi(z)>0$ in the $\mathbb{D}$.
Let \(z=x+iy\) with \(x^2+y^2<1\).
Set
\[
w=1+\frac z2 = a+ib,
\qquad
a=1+\frac x2,
\quad
b=\frac y2.
\]
Then, we have 
\[
\varphi(z)=w^2=(a+ib)^2=(a^2-b^2)+i(2ab),
\]
so
\[
\operatorname{Re}\,\varphi(z)=a^2-b^2
=\left(1+\frac x2\right)^2-\left(\frac y2\right)^2.
\]
Since \(|z|<1\), we have
$
x^2+y^2<1,$
hence
$
y^2<1-x^2.
$
Therefore, we have 
\[
\operatorname{Re}\,\varphi(z)
>
\left(1+\frac x2\right)^2-\frac{1-x^2}{4}, \quad\text{for} \quad z\in \mathbb{D}.
\]
A simple computation shows that 
\[
\left(1+\frac x2\right)^2=1+x+\frac{x^2}{4},
\]
so
\[
\left(1+\frac x2\right)^2-\frac{1-x^2}{4}
=\frac34+x+\frac{x^2}{2}.
\]
Completing the square gives
\[
\frac{x^2}{2}+x+\frac34
=\frac12(x+1)^2+\frac14
\ge \frac14.
\] 
Consequently, we have 
\[
\operatorname{Re}\,\varphi(z)>\frac14 
\qquad \text{for } z\in\mathbb D.
\]
Therefore, $\varphi(z)$ satisfies all the conditions of a Ma--Minda function.\\[2mm]
We denote by $\mathcal{C}(\varphi)$ the class of Ma--Minda convex functions subordinate to $\varphi(z).$ \textit{i.e.},
\[
\mathcal{C}(\varphi)
:= \left\{\, f \in \mathcal{A} : 1+z\frac{ f''(z)}{f'(z)} \prec \varphi(z) \,\right\}.
\]
%%%%%%%%%%%%%%%%%%%%%
%	\section{Integral representation of the class}
Now we provide the integral representation of the functions in the class $\mathcal{C}(\varphi).$
Let \(f\in \mathcal{C}(\varphi)\), then
\begin{align}\label{eq:2.2}
1+\frac{z f''(z)}{f'(z)} \prec \varphi(z).
\end{align}
By subordination, there exists a Schwarz function \(w:\mathbb{D}\to\mathbb{D}\),
analytic with \(w(0)=0\) and \(|w(z)|<1\), such that
\begin{align}\label{eq:2.3}
1+\frac{z f''(z)}{f'(z)} &= \varphi(w(z)), \qquad z\in \mathbb{D}.
\end{align}
Therefore
\[
\frac{f''(z)}{f'(z)}=\frac{\varphi(w(z))-1}{z}.
\]
Since \(f'(0)=1\) and \(f'(z)\neq 0\) in \(\mathbb{D}\), an analytic branch of \(\log f'(z)\) exists in \(\mathbb{D}\). Thus
\[
(\log f'(z))'=\frac{f''(z)}{f'(z)}=\frac{\varphi(w(z))-1}{z}.
\]
Integrating from \(0\) to \(z\), we obtain the general integral representation
\begin{align}
\log f'(z)&=\int_{0}^{z} \frac{\varphi(w(t))-1}{t}\,dt, \label{eq:2.4}\\
f(z)&=\int_{0}^{z} \exp\!\left(\int_{0}^{\zeta} \frac{\varphi(w(t))-1}{t}\,dt\right)\, d\zeta. \label{eq:2.5}
\end{align}
In particular, if we choose the special Schwarz function \(w(t)=t\), then we obtain the distinguished function \(\widetilde{f}\) given by
\begin{align}
\log \widetilde{f}^\prime(z)&=\int_{0}^{z} \frac{\varphi(t)-1}{t}\,dt, \label{eq:2.6}\\
\widetilde{f}(z)&=\int_{0}^{z} \exp\!\left(\int_{0}^{\zeta} \frac{\varphi(t)-1}{t}\,dt\right)\, d\zeta. \label{eq:2.7}
\end{align}
%%%%%%%%%%%%%%%%%%%
%%%%%%%%%%%%%%%%%%%%%%%
\section{Hankel determinants}

In 1966, Pommerenke \cite{Pommerenke1966} introduced the Hankel determinants for the class $\mathcal{S}$, and this topic has subsequently been studied by many authors. For a function $f \in \mathcal{A}$ with coefficients $\{a_n\}$, the $q$th Hankel determinant is defined by
\begin{equation}\label{eq:3.1}
H_q(n)=
\begin{vmatrix}
a_n & a_{n+1} & \cdots & a_{n+q-1} \\
a_{n+1} & a_{n+2} & \cdots & a_{n+q} \\
\vdots & \vdots & \ddots & \vdots \\
a_{n+q-1} & a_{n+q} & \cdots & a_{n+2q-2}
\end{vmatrix}
\qquad q,n\in\mathbb N.
\end{equation}
For special choices of $n$ and $q$, one may obtain well-known functionals, for instance, the classical Fekete--Szeg\"o functional, is defined as
\begin{align}\label{eq:3.2}
H_2(2)=a_2 a_4-a_3^{\,2}.
\end{align}
For our case $a_1=1$, the third Hankel determinant is
\begin{equation}\label{eq:3.3}
H_3(1)=
\begin{vmatrix}
1 & a_2 & a_3 \\
a_2 & a_3 & a_4 \\
a_3 & a_4 & a_5
\end{vmatrix}
= a_3(a_2 a_4-a_3^2)-a_4(a_4-a_2 a_3)+a_5(a_3-a_2^2).
\end{equation}

Let $\mathcal{P}$ denote the class of analytic functions $p$ defined on $\mathbb{D}$ such that 
\begin{equation}\label{eq:3.4}
p(z)=1+\sum_{n=1}^{\infty} p_n z^n
\end{equation}
with $\operatorname{Re} p(z)>0$ in $\mathbb{D}$. The following lemma 
(see \cite{LiberaZlotkiewicz1982, KwonLeckoSim2018}) will be useful for the next proof.

\begin{lem}\label{lemma1}
Let $p\in \mathcal{P}$ be given by \eqref{eq:3.4} with $p_1\ge 0$. Then
\begin{align*}
2p_2&=p_1^2+\gamma\bigl(4-p_1^2\bigr),\\[2mm]
4p_3&=p_1^3+2\bigl(4-p_1^2\bigr)p_1\gamma
-\bigl(4-p_1^2\bigr)p_1\gamma^2
+2\bigl(4-p_1^2\bigr)\bigl(1-|\gamma|^2\bigr)\eta,\\[2mm]
8p_4&=p_1^4+\bigl(4-p_1^2\bigr)\gamma\bigl(p_1^2(\gamma^2-3\gamma+3)+4\gamma\bigr)\\[1mm]
&\quad
-4\bigl(4-p_1^2\bigr)\bigl(1-|\gamma|^2\bigr)
\left(p_1(\gamma-1)\eta+\overline{\gamma}\,\eta^2-(1-|\eta|^2)\rho\right),
\end{align*}
for some $\gamma,\eta,\rho$ such that $|\gamma|\le 1$, $|\eta|\le 1$, and $|\rho|\le 1$.
\end{lem}
Next, we recall the following well-known result by Choi \textit{et al.} \cite{ChoiKimSugawa2007}, which will be useful for our main theorem.
\begin{lem}\label{lemma2} {\cite{ChoiKimSugawa2007}}
Let $A,B,C\in\mathbb R$ and define
\[
Y(A,B,C):=\max_{z\in\overline{\mathbb D}}\Bigl(|A+Bz+Cz^{2}|+1-|z|^{2}\Bigr).
\]
\begin{itemize}
\item[(i)] If $AC\ge 0$, then
\[
Y(A,B,C)=
\begin{cases}
	|A|+|B|+|C|, & \text{if } |B|\ge 2(1-|C|),\\[6pt]
	1+|A|+\dfrac{B^{2}}{4(1-|C|)}, & \text{if } |B|<2(1-|C|).
\end{cases}
\]
\item[(ii)] If $AC<0$, then
\[
Y(A,B,C)=
\begin{cases}
	1-|A|+\dfrac{B^{2}}{4(1-|C|)}, & \text{if } -4AC(C^{2}-1)\le B^{2}\ \text{and } |B|<2(1-|C|),\\[8pt]
	1+|A|+\dfrac{B^{2}}{4(1+|C|)}, & \text{if } B^{2}<\min\!\bigl\{\,4(1+|C|)^{2},-4AC(C^{2}-1)\bigr\},\\[8pt]
	R(A,B,C), & \text{otherwise,}
\end{cases}
\]
where,
\[
R(A,B,C)=
\begin{cases}
	|A|+|B|+|C|, & \text{if } |C|\bigl(|B|+4|A|\bigr)\le |AB|,\\[6pt]
	-|A|+|B|+|C|, & \text{if } |AB|\le |C|\bigl(|B|-4|A|\bigr),\\[10pt]
	\bigl(|A|+|C|\bigr)\sqrt{\,1-\dfrac{B^{2}}{4AC}\,}, & \text{otherwise.}
\end{cases}
\]
\end{itemize}
\end{lem}
We now proceed to prove the main results of this paper.
\begin{thm}\label{thm4.2}
Let $f\in \mathcal{C}(\varphi)$. Then
$
|H_2(2)|\le {1}/{36}.
$
The estimate is sharp.
\end{thm}

\begin{pf}
Let $f\in \mathcal{C}(\varphi)$ be given by 
\[
1+z\frac{f''(z)}{f'(z)}\prec \varphi(z), \qquad z\in \mathbb{D}.
\]
To determine the second Hankel determinant, we compute
\begin{align}\label{eq:3.5}
H_2(2)=a_2 a_4-a_3^2.
\end{align}
We first note that the class $\mathcal{C}(\varphi)$ is invariant under rotations. Indeed, if
\[
f_\theta(z):=e^{-i\theta}f(e^{i\theta}z),
\]
then \(f_\theta\in \mathcal{C}(\varphi)\), and
\[
H_2(2;f_\theta)=e^{2i(2)\theta}H_2(2;f)=e^{4i\theta}H_2(2;f).
\]
Hence \(|H_2(2;f_\theta)|=|H_2(2;f)|\), and therefore, without loss of generality, we may assume that the first coefficient of the associated Carath\'eodory function is nonnegative real.
Since $f\in \mathcal{C}(\varphi)$, there exists a Schwarz function $\omega(z)$ satisfying
\[
\omega(z)=\sum_{k=1}^{\infty} \omega_k z^k
\]
such that
\begin{equation}\label{eq:3.6}
1+\frac{z f''(z)}{f'(z)}
=
1+\omega(z)+\frac{1}{4}\omega(z)^2
=
\varphi(\omega(z)),
\qquad z\in \mathbb{D}.
\end{equation}
Let
\[
\omega(z)=\frac{p(z)-1}{p(z)+1},
\text{ where }
p(z)=1+p_1 z+p_2 z^2+p_3 z^3+p_4 z^4+\cdots \in \mathcal{P}.
\]
By the rotational property above, we may assume that $p_1\ge 0$. Since $p\in\mathcal{P}$, the Carath\'eodory estimate gives $|p_1|\le 2$, and therefore
$
0\le p_1\le 2.
$
After substituting the expressions for $\omega(z)$, $p(z)$, and
\[
f(z)=z+a_2 z^2+a_3 z^3+a_4 z^4+a_5 z^5+\cdots
\]
into \eqref{eq:3.6}, and comparing  coefficients of both sides, we obtain
\begin{equation}\label{eq:3.7}
\left.
\begin{aligned}
	a_2 &= \frac{p_1}{4}, \\[6pt]
	a_3 &= \frac{1}{96}\left(p_1^2+8p_2\right), \\[6pt]
	a_4 &= \frac{-p_1^3+32p_3}{768}, \\[6pt]
	a_5 &= \frac{
		11p_1^4
		-48p_1^2p_2
		-64p_1p_3
		-96p_2^2
		+768p_4
	}{30720}.
\end{aligned}
\right\}
\end{equation}
Substituting \eqref{eq:3.7} into \eqref{eq:3.5}, we obtain
\begin{align}\label{eq:3.8}
2304\,H_2(2)
=
-p_1^4-4p_1^2p_2-16p_2^2+24p_1p_3.
\end{align}
We now apply the parametric representation from Lemma~\ref{lemma1}, expressing $p_2$ and $p_3$ in terms of $p_1,\gamma,\eta$, namely
\[
2p_2=p_1^2+\gamma(4-p_1^2),
\]
and
\[
4p_3
=
p_1^3
+2(4-p_1^2)p_1\gamma
-(4-p_1^2)p_1\gamma^2
+2(4-p_1^2)(1-|\gamma|^2)\eta.
\]
Substituting these into \eqref{eq:3.8}, we obtain
\begin{align}\label{eq:3.9}
2304\,H_2(2)
&=
-p_1^4
-2p_1^2\bigl(p_1^2+\gamma(4-p_1^2)\bigr)
-4\bigl(p_1^2+\gamma(4-p_1^2)\bigr)^2 \notag\\
&\quad
+6p_1\Bigl(
p_1^3
+2\gamma p_1(4-p_1^2)
-\gamma^2 p_1(4-p_1^2)
+2\eta(1-|\gamma|^2)(4-p_1^2)
\Bigr).
\end{align}
Further, we can write
\[
2304\,H_2(2)=\Phi(p_1,\gamma,\eta),
\]
where $|\gamma|\le 1$, $|\eta|\le 1$, and $\Phi$ is of the form
\begin{align}\label{eq:3.10}
\Phi(p_1,\gamma,\eta)
=
A+B\gamma+C\gamma^2+\eta D_0(1-|\gamma|^2),
\end{align}
with
\begin{align*}
A&=-p_1^4,\\[4pt]
B&=-2p_1^2(-4+p_1^2)=2p_1^2(4-p_1^2),\\[4pt]
C&=2(-32+4p_1^2+p_1^4),
\end{align*}
and
\begin{align}\label{eq:3.11}
D_0:=12p_1(4-p_1^2)\ge 0.
\end{align}
We separate the endpoint cases $p_1=0$ and $p_1=2$.\\[2mm]
If $p_1=0$, then \eqref{eq:3.9} gives
\[
2304\,H_2(2)\big|_{p_1=0}=-64\,\gamma^2,
\]
and hence
\begin{equation}\label{eq:3.12}
|2304\,H_2(2)|\le 64.
\end{equation}
If $p_1=2$, then \eqref{eq:3.9} gives
\[
2304\,H_2(2)\big|_{p_1=2}=-16,
\qquad\text{that is,}\qquad
H_2(2)\big|_{p_1=2}=-\frac{1}{144}.
\]
In particular,
\begin{equation}\label{eq:3.13}
|2304\,H_2(2)|=16<64.
\end{equation}
We now consider the remaining case
$
0<p_1<2.
$
Then $D_0>0$, and for fixed $p_1$ and $\gamma$, if we write
\[
Z(\gamma):=A+B\gamma+C\gamma^2,
\]
we have
\[
\max_{|\eta|\le 1}\bigl|Z(\gamma)+\eta D_0(1-|\gamma|^2)\bigr|
=
|Z(\gamma)|+D_0(1-|\gamma|^2).
\]
Therefore
\begin{align}\label{eq:3.14}
\max_{|\gamma|\le 1,\,|\eta|\le 1}|2304\,H_2(2)|
&=
D_0\max_{|\gamma|\le 1}\Bigl(
|\widetilde{A}+\widetilde{B}\gamma+\widetilde{C}\gamma^2|
+1-|\gamma|^2
\Bigr),
\end{align}
where
\[
\widetilde{A}:=\frac{A}{D_0},
\qquad
\widetilde{B}:=\frac{B}{D_0},
\qquad
\widetilde{C}:=\frac{C}{D_0}.
\]
In particular,
\[
\widetilde{A}=\frac{-p_1^4}{12p_1(4-p_1^2)},
\qquad
\widetilde{B}=\frac{p_1}{6},
\qquad
\widetilde{C}=-\frac{8+p_1^2}{6p_1}.
\]
Since $0<p_1<2$, these are all real. Moreover,
\begin{align*}
\widetilde{A}\widetilde{C}
=
\frac{p_1^2(8+p_1^2)}{72(4-p_1^2)}
\ge 0,
\end{align*}
so Lemma~\ref{lemma2}(i) applies.\\[2mm]
Next, for $0<p_1<2$ we have
\[
|\widetilde{C}|=\frac{8+p_1^2}{6p_1}>1,
\]
because
\[
8+p_1^2-6p_1=(p_1-2)(p_1-4)>0
\qquad \text{for } 0<p_1<2.
\]
Hence
\[
2(1-|\widetilde{C}|)<0,
\]
and therefore the condition
\[
|\widetilde{B}|\ge 2(1-|\widetilde{C}|)
\]
is  satisfied. Thus, by lemma~\ref{lemma2}(i), we obtain
\begin{align}\label{eq:3.15}
\max_{|\gamma|\le 1,\,|\eta|\le 1}|2304\,H_2(2)|
&=
D_0\bigl(|\widetilde{A}|+|\widetilde{B}|+|\widetilde{C}|\bigr) \notag\\
&=
|A|+|B|+|C|.
\end{align}	
Now, for $0<p_1<2$, we have
\[
A=-p_1^4<0,
\qquad
B=2p_1^2(4-p_1^2)>0,
\]
and also
\[
C=2(-32+4p_1^2+p_1^4)<0,
\]
since
\[
p_1^4+4p_1^2<16+16=32.
\]
Therefore,
\begin{align}\label{eq:3.16}
|A|+|B|+|C|
&=
p_1^4+2p_1^2(4-p_1^2)-2(-32+4p_1^2+p_1^4) \notag\\
&=
p_1^4+(8p_1^2-2p_1^4)+(64-8p_1^2-2p_1^4) \notag\\
&=
64-3p_1^4
<64.
\end{align}
Hence, for every $p_1\in(0,2)$,
\begin{equation}\label{eq:3.17}
|2304\,H_2(2)|<64.
\end{equation}
Combining \eqref{eq:3.12}, \eqref{eq:3.13}, and \eqref{eq:3.17}, we conclude that
\[
2304\,|H_2(2)|\le 64,
\quad p_1\in[0,2].
\]
Therefore,
\[
|H_2(2)|\le \frac{64}{2304}=\frac{1}{36}.
\]
It remains to show sharpness. Take
\[
\omega(z)=z^2,
\]
and denote the corresponding function in $\mathcal{C}(\varphi)$ by $\widetilde{f}$. Then
\[
1+z\frac{\widetilde{f}''(z)}{\widetilde{f}'(z)}=\varphi(z^2).
\]
Therefore,
\begin{align*}
\widetilde{f}(z)
&=
\int_0^z
\exp\!\left(
\int_0^\zeta \frac{\varphi(t^2)-1}{t}\,dt
\right)d\zeta \\[4pt]
&=
z+\frac{1}{6}z^3+\frac{3}{80}z^5+\frac{5}{672}z^7+\frac{19}{13824}z^9+\cdots.
\end{align*}
Thus, we have 
\[
a_2=0,\qquad a_3=\frac{1}{6},\qquad a_4=0,
\]
Hence, we obtain
\[
H_2(2)=a_2 a_4-a_3^2
=
0\cdot 0-\left(\frac{1}{6}\right)^2
=
-\frac{1}{36}.
\]
Thus the estimate is sharp.
\end{pf}\\
%%%%%%%%%%%%%%
The following lemma \cite{ProkhorovSzynal1981} will be useful for the proof of our next theorem.
%%%%%%%%%%%%%%
\begin{lem}\label{lemma3}
Let
$
\omega(z)=\sum_{n=1}^{\infty}c_n z^n
$
be a Schwarz function, that is, $\omega$ is analytic in $\mathbb D$,
$\omega(0)=0$, and $|\omega(z)|<1$ for all $z\in\mathbb D$.
If $c_1\ge 0$, then there exist complex numbers $\gamma,\eta,\rho$
with
$
|\gamma|\le 1,\, |\eta|\le 1,\, |\rho|\le 1,
$
such that
\begin{align*}
c_2 &= (1-c_1^2)\gamma,\\[2mm]
c_3 &= (1-c_1^2)\bigl(\eta(1-|\gamma|^2)-c_1\gamma^2\bigr),\\[2mm]
c_4 &= (1-c_1^2)\Bigl(
c_1^2\gamma^3
-(1-|\gamma|^2)\bigl(2c_1\gamma\eta+\overline{\gamma}\eta^2\bigr)
+(1-|\gamma|^2)(1-|\eta|^2)\rho
\Bigr).
\end{align*}
\end{lem}
%%%%%%%%%%%%%%
%%%%%%%%%%%%%%%%%
\begin{thm}
Let $f\in \mathcal{C}(\varphi)$. Then $|H_3(1)|\leq {1}/{144}$, where $H_3(1)$ denotes the third Hankel determinant. The result is sharp.
\end{thm}
\begin{pf}
Let $f\in \mathcal{C}(\varphi)$ be given by 
\[
1+z\frac{f''(z)}{f'(z)} \prec \varphi(z), \quad z\in \mathbb{D}.
\]
The third Hankel determinant is given  by
\begin{align}\label{eq:3.18}
H_3(1) =
\begin{vmatrix}
	1 & a_2 & a_3 \\
	a_2 & a_3 & a_4 \\
	a_3 & a_4 & a_5
\end{vmatrix}
= a_3(a_2 a_4 - a_3^2) - a_4(a_4 - a_2 a_3) + a_5(a_3 - a_2^2).
\end{align}
For the third Hankel determinant we use the coefficients of the Schwarz function together with the relation between the Carath\'eodory and Schwarz functions:
\[
\omega(z):=\frac{p(z)-1}{p(z)+1}, \quad z\in \mathbb{D},
\]
where
\[
\omega(z)=\sum_{n=1}^{\infty}c_n z^n
= c_1 z+c_2 z^2+c_3 z^3+\cdots.
\]
Now comparing the coefficients in \eqref{eq:3.6}, while using  $\omega(z)$ in the series form as above, we obtain
\begin{align}\label{eq:3.19}
\left.
\begin{aligned}
	a_{2} & = \frac{c_{1}}{2},\\[1mm]
	a_{3} & = \frac{1}{24}\left(5c_{1}^{2}+4c_{2}\right),\\[1mm]
	a_{4} & = \frac{1}{24}\left(-2c_{1}^{3}+c_{1}c_{2}+\frac{3}{4}c_{1}\left(5c_{1}^{2}+4c_{2}\right)+2c_{3}\right),\\[1mm]
	a_{5} & = \frac{43c_{1}^{4}+184c_{1}^{2}c_{2}+72c_{2}^{2}+176c_{1}c_{3}+96c_{4}}{1920}.
\end{aligned}
\right\}
\end{align}
Now putting \eqref{eq:3.19} in \eqref{eq:3.18}, we obtain
\begin{align}\label{eq:3.20}
69120 H_{3}(1)&=-7c_{1}^{6}+42c_{1}^{4}c_{2}+96c_{1}^{3}c_{3}+96c_{1}c_{2}c_{3}
-12c_{1}^{2}\left(17c_{2}^{2}+12c_{4}\right)\notag\\[2mm]
&\quad+16\left(7c_{2}^{3}-30c_{3}^{2}+36c_{2}c_{4}\right).
\end{align}
Since the class is rotation invariant, we may take $c_{1} \geq 0$. Moreover, since \(\omega\) is a Schwarz function, the Schwarz lemma gives \(|c_1|\le 1\). Hence
\[
0\le c_1\le 1.
\]
Now we use the parametric relation between the coefficients of the Schwarz functions in Lemma~\ref{lemma3} to obtain
\begin{align}
69120\,  H_{3}(1)&=-7c_{1}^{6}
+42\gamma\,c_{1}^{4}(1-c_{1}^{2})+96c_{1}^{3}\bigl(\eta(1-|\gamma|^{2})-\gamma^{2}c_{1}\bigr)(1-c_{1}^{2})\notag\\
&
\quad+96\gamma\,c_{1}\bigl(\eta(1-|\gamma|^{2})-\gamma^{2}c_{1}\bigr)(1-c_{1}^{2})^{2} \notag\\
&\quad
-12c_{1}^{2}\Bigl(
17\gamma^{2}(1-c_{1}^{2})^{2}
+12(1-c_{1}^{2})\notag\\
&\qquad
\Bigl(\gamma^{3}c_{1}^{2}
-(1-|\gamma|^{2})\bigl(-\rho(1-|\eta|^{2})+2\gamma\eta c_{1}+\eta^{2}\overline{\gamma}\bigr)\Bigr)
\Bigr) \notag\\
&\quad
+16\Bigl(
-30\bigl(\eta(1-|\gamma|^{2})-\gamma^{2}c_{1}\bigr)^{2}(1-c_{1}^{2})^{2}
+7\gamma^{3}(1-c_{1}^{2})^{3} \notag\\
&\quad
+36\gamma(1-c_{1}^{2})^{2}\Bigl(\gamma^{3}c_{1}^{2}
-(1-|\gamma|^{2})\bigl(-\rho(1-|\eta|^{2})+2\gamma\eta c_{1}+\eta^{2}\overline{\gamma}\bigr)\Bigr)
\Bigr).
\end{align}
We may re-write the above equation in the form
\begin{align}\label{eq:3.22}
69120\,  H_3(1)=A_{1}+B_{1} \eta+C_{1} \eta^2+D_{1} \rho,
\qquad  \mbox{where} \quad |\gamma|\le 1,\ |\eta|\le 1,\ |\rho|\le 1.
\end{align}
where
\begin{align*}
A_{1}&=-7c_{1}^{6}-42\gamma\,c_{1}^{4}(-1+c_{1}^{2})
+96\gamma^{4}c_{1}^{2}(-1+c_{1}^{2})^{2}
-12\gamma^{2}c_{1}^{2}\left(17-26c_{1}^{2}+9c_{1}^{4}\right)\\
&\quad-16\gamma^{3}\left(-7+27c_{1}^{2}-24c_{1}^{4}+4c_{1}^{6}\right),\\[2mm]
B_{1}&=96\left(-1+|\gamma|^{2}\right)c_{1}(-1+c_{1}^{2})
\left(\gamma+c_{1}^{2}+2\gamma c_{1}^{2}+2\gamma^{2}(-1+c_{1}^{2})\right),\\[2mm]
C_{1}&=-48\left(-1+|\gamma|^{2}\right)(-1+c_{1}^{2})
\Bigl(-10(-1+c_{1}^{2})+10|\gamma|^{2}(-1+c_{1}^{2})
-3\bigl(c_{1}^{2}+4\gamma(-1+c_{1}^{2})\bigr)\overline{\gamma}\Bigr),\\[2mm]
D_{1}&=144\left(-1+|\gamma|^{2}\right)\left(-1+|\eta|^{2}\right)(-1+c_{1}^{2})
\left(c_{1}^{2}+4\gamma(-1+c_{1}^{2})\right).
\end{align*}
Taking modulus on both sides in \eqref{eq:3.22}, writing
\[
p_1:=c_1,\qquad x:=|\gamma|\in[0,1],\qquad y:=|\eta|\in[0,1],
\]
and using $|\rho|\le 1$ together with the triangle inequality, we obtain
\begin{align*}
69120\, |H_{3}(1)| \leq H(p_1, x, y),
\end{align*}
where
\begin{align}\label{eq:3.23}
H(p_1,x,y)
&= 7p_1^{6}
+ x\Bigl(42p_1^{4}(1-p_1^{2})\Bigr)
+ x^{2}\Bigl(12p_1^{2}\bigl(17-26p_1^{2}+9p_1^{4}\bigr)\Bigr)\notag\\
&\quad
+ x^{3}\Bigl(16\bigl(7(1-p_1^{2})+20p_1^{2}(1-p_1^{2})+4p_1^{4}(1-p_1^{2})\bigr)\Bigr)
+ x^{4}\Bigl(96p_1^{2}(-1+p_1^{2})^{2}\Bigr)\notag\\
&\quad
+ y\Bigl(96(1-x^{2})p_1(1-p_1^{2})\bigl(x+p_1^{2}+2xp_1^{2}+2x^{2}(1-p_1^{2})\bigr)\Bigr)\notag\\
&\quad
+ y^{2}\Bigl(48(1-x^{2})(1-p_1^{2})\bigl(10(1-x^{2})(1-p_1^{2})
+3\bigl(p_1^{2}+4x(1-p_1^{2})\bigr)x\bigr)\Bigr)\notag\\
&\quad
+ 144(1-x^{2})(1-y^{2})(1-p_1^{2})\bigl(p_1^{2}+4x(1-p_1^{2})\bigr).
\end{align}
over the cuboid $\mathcal{D}:=\{(p_1, x, y):(p_1, x, y)\in [0, 1] \times [0, 1] \times [0, 1]\}$, which we denote by $\mathcal{D}$. Our aim is to maximize the expression $H(p_1, x, y)$ over this region $\mathcal{D}$.\\[2mm]
It is easy to see that the coefficient of $y$ remains non-negative in $\mathcal{D}$. Therefore, we take $y=1$ in the term
\[
y\Bigl(96(1-x^{2})p_1(1-p_1^{2})\bigl(x+p_1^{2}+2xp_1^{2}+2x^{2}(1-p_1^{2})\bigr)\Bigr)
\]
to make the computation easier. Furthermore, we define a new function $H_{1}(p_1, x, y)$ such that $H(p_1, x, y) \leq H_{1}(p_1, x, y)$ in $\mathcal{D}$. Precisely,
\begin{align}\label{eq:3.24}
H_1(p_1,x,y)
&= 7p_1^{6}
+ x\Bigl(42p_1^{4}(1-p_1^{2})\Bigr)
+ x^{2}\Bigl(12p_1^{2}\bigl(17-26p_1^{2}+9p_1^{4}\bigr)\Bigr)\notag\\
&\quad
+ x^{3}\Bigl(16\bigl(7(1-p_1^{2})+20p_1^{2}(1-p_1^{2})+4p_1^{4}(1-p_1^{2})\bigr)\Bigr)
+ x^{4}\Bigl(96p_1^{2}(-1+p_1^{2})^{2}\Bigr)\notag\\
&\quad
+ 96(1-x^{2})p_1(1-p_1^{2})\bigl(x+p_1^{2}+2xp_1^{2}+2x^{2}(1-p_1^{2})\bigr)\notag\\
&\quad
+ y^{2}\Bigl(48(1-x^{2})(1-p_1^{2})\bigl(10(1-x^{2})(1-p_1^{2})
+3\bigl(p_1^{2}+4x(1-p_1^{2})\bigr)x\bigr)\Bigr)\notag\\
&\quad
+ 144(1-x^{2})(1-y^{2})(1-p_1^{2})\bigl(p_1^{2}+4x(1-p_1^{2})\bigr).
\end{align}
Since $H_{1}(p_{1}, x, y)$ is quadratic in $y$ and the coefficient of $y$ is $0$, the maximum occurs at one of the endpoints $y=0$ or $y=1$. We handle these two cases separately.\\[2mm]
Set $G_{1}(p_1, x):=H_{1}(p_{1}, x, 1)$ and $G_{2}(p_1, x):=H_{1}(p_{1}, x, 0)$. \\[2mm]
For the case when $y=1$ we prove that the function $G_{1}(p_1, x)\le 480$ in $[0, 1]\times[0, 1]$. For the same purpose we take $480-G_{1}(p_{1}, x)$ and prove that it remains non-negative in $[0, 1] \times[0, 1]$.\\
\[
\begin{aligned}
G_1(p_1,x)=\;&7 p_1^6 + 42 p_1^4 (1-p_1^2)x + 12 p_1^2(17-26p_1^2+9p_1^4)x^2 \\
&+16\bigl(7(1-p_1^2)+20p_1^2(1-p_1^2)+4p_1^4(1-p_1^2)\bigr)x^3 \\
&+96p_1^2(-1+p_1^2)^2x^4 \\
&+96p_1(1-p_1^2)(1-x^2)\bigl(p_1^2+x+2p_1^2x+2(1-p_1^2)x^2\bigr) \\
&+48(1-p_1^2)(1-x^2)\bigl(3x(p_1^2+4(1-p_1^2)x)+10(1-p_1^2)(1-x^2)\bigr).
\end{aligned}
\]
Taking
\[
p:=p_1,
\qquad
F(p,x):=480-G_1(p,x),
\]
we show that
\[
F(p,x)\ge 0
\qquad \text{for all }(p,x)\in[0,1]\times[0, 1],
\]
with equality only at $(0,0)$.
We use the fact that if a polynomial $H(u,v)$ is written on the unit square in the  Bernstein basis
\[
H(u,v)=\sum_{i=0}^{m}\sum_{j=0}^{n} c_{ij}\,B_i^m(u)B_j^n(v),
\]
then
\[
\min_{i,j} c_{ij}\le H(u,v)\le \max_{i,j} c_{ij}
\qquad 0\le u,v\le 1,
\]
because all Bernstein basis functions are nonnegative and their sum is $1$.
Since $G_1$ has bidegree $(6,4)$, we use the  Bernstein basis of bidegree $(6,4)$.\\[2mm]
First we divide the square \([0, 1]\times[0, 1]\) into
\begin{align*}
Q_1&=\left[0,\frac12\right]\times\left[0,\frac12\right],\qquad
Q_2=\left[0,\frac12\right]\times\left[\frac12,1\right],\\[3mm]
Q_3&=\left[\frac12,1\right]\times\left[0,\frac12\right],\qquad
Q_4=\left[\frac12,1\right]\times\left[\frac12,1\right].
\end{align*}
%%%%%%%%%%
For each subrectangle \(Q_k=[a,b]\times[c,d]\), we first transfer the problem to the unit square by writing
\[
p=a+(b-a)u,\qquad x=c+(d-c)v,\qquad (u,v)\in[0, 1]\times[0, 1].
\]
The polynomial \(G_1(p,x)\) then becomes a polynomial in \((u,v)\), which we express in the Bernstein basis of bidegree \((6,4)\),
\[
G_1(a+(b-a)u,\;c+(d-c)v)
=
\sum_{i=0}^{6}\sum_{j=0}^{4} b_{ij}\,B_i^6(u)B_j^4(v).
\]
Because
\[
B_i^6(u)\ge 0,\qquad B_j^4(v)\ge 0,
\qquad
\sum_{i=0}^{6}\sum_{j=0}^{4} B_i^6(u)B_j^4(v)=1,
\]
it follows that the value of the polynomial on the rectangle is bounded above by the largest Bernstein coefficient \(b_{ij}\) arising from that representation. Performing this calculation on each of the four rectangles gives
\begin{align*}
G_1(p,x) &\le 481 \qquad \text{on } Q_1,\\[2mm]
G_1(p,x) &\le 398 \qquad \text{on } Q_2,\\[2mm]
G_1(p,x) &\le \frac{75535}{256} \quad \text{on } Q_3,\\[2mm]
G_1(p,x) &\le \frac{18463}{64} \quad \text{on } Q_4.
\end{align*}
%%%%%%%%%%
Since
\begin{align*}
398<480,\qquad \frac{75535}{256}<480,\qquad \frac{18463}{64}<480,
\end{align*}
it follows that
\[
F(p,x)>0
\qquad \text{on }Q_2\cup Q_3\cup Q_4.
\]
Thus only $Q_{1}$ can contain a zero of $F$.
Now subdivide
\[
Q_1=\left[0,\frac12\right]\times\left[0,\frac12\right]
\]
into
\begin{align*}
Q_{11}&=\left[0,\frac14\right]\times\left[0,\frac14\right],\qquad
Q_{12}=\left[0,\frac14\right]\times\left[\frac14,\frac12\right],\\[3mm]
Q_{13}&=\left[\frac14,\frac12\right]\times\left[0,\frac14\right],\qquad
Q_{14}=\left[\frac14,\frac12\right]\times\left[\frac14,\frac12\right].
\end{align*}
Applying the Bernstein method once more on these four rectangles, we obtain
\begin{align*}
G_1(p,x) &\le \frac{1921}{4} \qquad \text{on } Q_{11},\\[2mm]
G_1(p,x) &\le \frac{14681}{32} \qquad \text{on } Q_{12},\\[2mm]
G_1(p,x) &\le \frac{13939571}{32768} \quad \text{on } Q_{13},\\[2mm]
G_1(p,x) &\le \frac{13594541}{32768} \quad \text{on } Q_{14}.
\end{align*}
Since
\begin{align*}
\frac{14681}{32}<480,\qquad
\frac{13939571}{32768}<480,\qquad
\frac{13594541}{32768}<480,
\end{align*}
we get
\[
F(p,x)>0
\qquad \text{on }Q_{12}\cup Q_{13}\cup Q_{14}.
\]
Therefore it remains only to show that
\[
F(p,x)\ge 0
\qquad \text{on }Q_{11}=\left[0,\frac14\right]\times\left[0,\frac14\right].
\]
We now consider the rectangle $Q_{11}$.
First expand $F(p_1, x)$ by 
\[
\begin{aligned}
F(p,x)=\;&960p^2-96px+384x^2 -96p^3-480p^4+96p^5-7p^6 \\
&+\bigl(-144p^2-96p^3+102p^4+192p^5+42p^6\bigr)x \\
&+\bigl(-192p-972p^2+480p^3+696p^4-288p^5-108p^6\bigr)x^2 \\
&+\bigl(-112+96p-64p^2+96p^3+112p^4-192p^5+64p^6\bigr)x^3 \\
&+\bigl(96+192p-288p^2-384p^3+288p^4+192p^5-96p^6\bigr)x^4.
\end{aligned}
\]
Let, 
\[
F(p,x)=Q(p,x)+R(p,x),
\]
where,
\[
Q(p,x)=960p^2-96px+384x^2
\]
and
\[
\begin{aligned}
R(p,x)=\;&-96p^3-480p^4+96p^5-7p^6 \\
&+\bigl(-144p^2-96p^3+102p^4+192p^5+42p^6\bigr)x \\
&+\bigl(-192p-972p^2+480p^3+696p^4-288p^5-108p^6\bigr)x^2 \\
&+\bigl(-112+96p-64p^2+96p^3+112p^4-192p^5+64p^6\bigr)x^3 \\
&+\bigl(96+192p-288p^2-384p^3+288p^4+192p^5-96p^6\bigr)x^4.
\end{aligned}
\]
Every monomial in $R(p, x)$ has total degree at least $3$.
We first estimate $Q(p, x)$ from below. Since
$
2px\le p^2+x^2,
$
we have
$
96px\le 48p^2+48x^2.
$
Therefore
\[
Q(p,x)\ge (960-48)p^2+(384-48)x^2
=912p^2+336x^2
\ge 336(p^2+x^2).
\]
Next we estimate $R(p, x)$ on $Q_{11}$, where
\[
0\le p\le \frac14,\qquad 0\le x\le \frac14.
\]
Let
\begin{align*}
R(p,x)=\sum_{i+j\ge 3} c_{ij}p^i x^j.
\end{align*}
For every monomial $p^i x^j$ with $i+j\ge 3$, we claim that
\begin{align*}
p^i x^j \le 4^{-(i+j-2)}(p^2+x^2).
\end{align*}
Indeed, if \(i\ge 2\), then
\begin{align*}
p^i x^j=p^2 p^{\,i-2}x^j\le 4^{-(i+j-2)}p^2
\le 4^{-(i+j-2)}(p^2+x^2).
\end{align*}
If $i=1$, then $j\ge 2$, and
\begin{align*}
px^j=x^2(px^{j-2})\le 4^{-(j-1)}x^2
=4^{-(i+j-2)}x^2
\le 4^{-(i+j-2)}(p^2+x^2).
\end{align*}
If \(i=0\), then \(j\ge 3\), and
\[
x^j\le 4^{-(j-2)}x^2
=4^{-(i+j-2)}x^2
\le 4^{-(i+j-2)}(p^2+x^2).
\]
Therefore, we have 
\begin{align*}
|R(p,x)|
\le
\left(\sum_{i+j\ge 3}|c_{ij}|\,4^{-(i+j-2)}\right)(p^2+x^2).
\end{align*}
Now group the coefficients according to the total degree \(i+j\).
The sums of the absolute values of the coefficients of total degree
\(3,4,5,6,7,8,9,10\) are respectively
\[
544,\quad 1740,\quad 934,\quad 1279,\quad 826,\quad 588,\quad 256,\quad 96.
\]
Hence
\[
\sum_{i+j\ge 3}|c_{ij}|\,4^{-(i+j-2)}
=
\frac{544}{4}
+\frac{1740}{16}
+\frac{934}{64}
+\frac{1279}{256}
+\frac{826}{1024}
+\frac{588}{4096}
+\frac{256}{16384}
+\frac{96}{65536}.
\]
A direct simplification gives
\[
\sum_{i+j\ge 3}|c_{ij}|\,4^{-(i+j-2)}
=
\frac{543349}{2048}.
\]
Therefore
\[
|R(p,x)|\le \frac{543349}{2048}(p^2+x^2).
\]
Combining the estimates for $Q$ and $R$, we get
\[
F(p,x)=Q(p,x)+R(p,x)
\ge 336(p^2+x^2)-\frac{543349}{2048}(p^2+x^2).
\]
Thus
\[
F(p,x)\ge
\left(336-\frac{543349}{2048}\right)(p^2+x^2)
=
\frac{144779}{2048}(p^2+x^2).
\]
Since
\[
\frac{144779}{2048}>0,
\]
it follows that
\[
F(p,x)\ge 0
\qquad \text{for all }(p,x)\in Q_{11},
\]
it is easy to see that
\begin{align*}
F(p,x)>0
\qquad \text{for }(p,x)\in Q_{11}\setminus\{(0,0)\}.
\end{align*}
Also, we have
\[
F(0,0)=480-G_1(0,0)=0.
\]
Combining this with the Bernstein bounds on the other rectangles, we conclude that
\[
F(p,x)\ge 0
\qquad \text{for all }(p,x)\in[0, 1]\times[0, 1],
\]
with equality only at \((0,0)\).
Equivalently,
\[
G_1(p_1,x)\le 480
\qquad \text{for all }(p_1,x)\in[0, 1]\times[0, 1],
\]
and equality holds only at
\[
(p_1,x)=(0,0).
\]
Therefore
\[
\max_{0\le p_1\le 1,\;0\le x\le 1}G_1(p_1,x)=480.
\]
Now we solve the case when 
$y=0$, precisely for the function $G_{2}(p_1, x):=H_{1}(p_1, x, 0)$.\\[2mm]
Setting $p_{1}=p$, and applying the Bernstein method to
\[
\begin{aligned}
G_2(p,x)=\;&7 p^6 + 42 p^4 (1-p^2)x + 12 p^2(17-26p^2+9p^4)x^2 \\
&+16\bigl(7(1-p^2)+20p^2(1-p^2)+4p^4(1-p^2)\bigr)x^3 \\
&+96p^2(-1+p^2)^2x^4 \\
&+144(1-p^2)\bigl(p^2+4(1-p^2)x\bigr)(1-x^2) \\
&+96p(1-p^2)(1-x^2)\bigl(p^2+x+2p^2x+2(1-p^2)x^2\bigr)
\end{aligned}
\]
on the rectangle
\[
[0,1]\times[0,1].
\]
Since the given rectangle is already the unit square, no affine transformation is needed. We first expand \(G_2(p,x)\) in the power basis:
\[
\begin{aligned}
G_2(p,x)=\;&\bigl(7p^6-96p^5-144p^4+96p^3+144p^2\bigr) \\
&+\bigl(-42p^6-192p^5+618p^4+96p^3-1152p^2+96p+576\bigr)x \\
&+\bigl(108p^6+288p^5-168p^4-480p^3+60p^2+192p\bigr)x^2 \\
&+\bigl(-64p^6+192p^5-832p^4-96p^3+1360p^2-96p-464\bigr)x^3 \\
&+\bigl(96p^6-192p^5-192p^4+384p^3+96p^2-192p\bigr)x^4.
\end{aligned}
\]
Now consider the Bernstein basis of bidegree \((6,4)\):
\[
B_i^6(p)=\binom{6}{i}p^i(1-p)^{6-i},
\qquad
B_j^4(x)=\binom{4}{j}x^j(1-x)^{4-j}.
\]
Then
\[
G_2(p,x)=\sum_{i=0}^{6}\sum_{j=0}^{4} c_{ij}\,B_i^6(p)B_j^4(x),
\]
where the Bernstein coefficients \(c_{ij}\) are determined by the power-basis expansion.

For the present polynomial, the Bernstein coefficient matrix \((c_{ij})\) is \\
\[
\begin{pmatrix}
0 & 144 & 288 & 316 & 112 \\[5mm]
0 & 148 & \dfrac{904}{3} & 340 & 112 \\[5mm]
\dfrac{48}{5} & \dfrac{712}{5} & \dfrac{4298}{15} & \dfrac{1022}{3} & \dfrac{2188}{15} \\[5mm]
\dfrac{168}{5} & \dfrac{666}{5} & \dfrac{1234}{5} & \dfrac{1566}{5} & \dfrac{1068}{5} \\[5mm]
\dfrac{336}{5} & \dfrac{1271}{10} & \dfrac{2917}{15} & \dfrac{7639}{30} & \dfrac{802}{3} \\[5mm]
80 & \dfrac{215}{2} & 127 & \dfrac{935}{6} & \dfrac{634}{3} \\[5mm]
7 & 7 & 7 & 7 & 7
\end{pmatrix}.
\]\\
We now use the fact that for \((p,x)\in[0, 1]\times[0, 1]\),
\[
B_i^6(p)\ge 0,\qquad B_j^4(x)\ge 0,
\]
and
\[
\sum_{i=0}^{6}\sum_{j=0}^{4} B_i^6(p)B_j^4(x)=1.
\]
Hence \(G_2(p,x)\) is a convex combination of the Bernstein coefficients. Therefore,
\[
\min_{0\le i\le 6,\;0\le j\le 4} c_{ij}
\le
G_2(p,x)
\le
\max_{0\le i\le 6,\;0\le j\le 4} c_{ij}
\]
for all $(p,x)\in[0, 1]\times[0, 1]$.
From the coefficient matrix above, we see that
\[
\max_{0\le i\le 6,\;0\le j\le 4} c_{ij}=\frac{1022}{3}.
\]
Consequently,
\[
G_2(p_1,x)\le \frac{1022}{3}
\qquad
\text{for all }0\le p_1\le 1,\ 0\le x\le 1.
\]
Thus the Bernstein method gives the upper bound
\[
\max_{0\le p_1\le 1,\;0\le x\le 1} G_2(p_1,x)\le \frac{1022}{3}\le 480.
\]
Hence, exhausting all the cases, we have
\[
69120\,|H_3(1)|\le 480.
\]
Therefore,
\begin{align*}
|H_3(1)|\le \frac{480}{69120}=\frac{1}{144}.
\end{align*}
Finally, we show that $\omega(z)=z^3$ provides an extremal function for the third Hankel determinant. In this case,
\[
1+z\frac{\widetilde{f}''(z)}{\widetilde{f}'(z)}=\varphi(z^3)
=1+z^3+\frac{z^6}{4}.
\]
Therefore,
\begin{align*}
\widetilde{f}(z)
&=\int_{0}^{z} \exp\!\left(\int_{0}^{\zeta} \frac{\varphi(t^3)-1}{t}\,dt\right)\, d\zeta,\\[2pt]
\widetilde{f}(z)
&=z + \frac{1}{12}z^4 + \frac{1}{72} z^7+\ldots .
\end{align*}
It is easy to check that this $\widetilde{f}$ gives $|H_3(1)|=\dfrac{1}{144}$, which proves the sharpness.
\end{pf}
\vspace{3mm}

\noindent\textbf{Compliance of Ethical Standards:}\\

\noindent\textbf{Conflict of interest.} The authors declare that there is no conflict  of interest regarding the publication of this paper.
\vspace{1.5mm}

\noindent\textbf{Data availability statement.}  Data sharing is not applicable to this article as no datasets were generated or analyzed during the current study.
\vspace{1.5mm}

\noindent\textbf{Authors contributions.} Both the authors have made equal contributions in reading, writing, and preparing the manuscript.\\

\noindent\textbf{Acknowledgment:} 
The second named author acknowledges financial support from the Council of Scientific and Industrial Research (CSIR), Government of India, through a CSIR Fellowship.
%%%%%%%%%%%%%%%%%

\end{document}